\documentclass[a4paper,11pt]{amsart}
\usepackage{amsmath}
\usepackage{amssymb}
\usepackage{amstext}
\usepackage{amsbsy}
\usepackage{amsopn}
\usepackage{amsthm}
\usepackage{amscd}
\usepackage{amsxtra}
\usepackage{amsfonts}

\usepackage{ifthen}
\usepackage{xspace}
\usepackage{fancyheadings}
\usepackage{layout}
\usepackage{hyperref}

\usepackage{xcolor}

\input xy
\xyoption{matrix}
\xyoption{arrow}
\xyoption{curve}
\newdir{ (}{{}*!/-5pt/@^{(}}

\newboolean{xlabels}
\newcommand{\xlabel}[1]{
                        \label{#1}
                        \ifthenelse{\boolean{xlabels}}
                                   {\marginpar[\hfill{\tiny #1}]{{\tiny #1}}}
                                   {}
                       }

\newcommand{\ZZ}{\mathbb{Z}}

\newcommand{\CC}{\mathbb{C}}

\newcommand{\QQ}{\mathbb{Q}}

\newcommand{\PP}{\mathbb{P}}

\newcommand{\FF}{\mathbb{F}}

\newcommand{\sV}{{\mathcal V}}

\newboolean{probleme}
\newcommand{\problem}[1]
           {\ifthenelse{\boolean{probleme}}
                       {{\color{red}(PROBLEM: #1)}}
                       {}
           }
\newboolean{zukuenftiges}
\newcommand{\zukunft}[1]
           {\ifthenelse{\boolean{zukuenftiges}}
                       {{\bf(AUSBAUM\"OGLICHKEIT: #1)\bf}}
                       {}
           }
\newboolean{extras}
\newcommand{\extra}[1]
           {\ifthenelse{\boolean{extras}}
                       {{\bf EXTRA #1 EXTRA\bf}}
                       {}
           }

\newboolean{ignore}
\newcommand{\ignore}[1]
           {\ifthenelse{\boolean{ignore}}
                       {{\bf IGNORE #1 IGNORE\bf}}
                       {}
           }

\DeclareMathOperator{\Img}{Im}

\DeclareMathOperator{\sing}{sing}

\theoremstyle{plain}
\begingroup

\newtheorem{thm}{Theorem}

\newtheorem{lem}[thm]{Lemma}

\newtheorem{prop}[thm]{Proposition}

\numberwithin{thm}{subsection} 
\endgroup

\newtheorem*{thm*}{Theorem}
\newtheorem*{conj*}{Conjecture}
\newtheorem*{verm*}{Vermutung}

\theoremstyle{definition}
\newtheorem{defn}[thm]{Definition}
\newtheorem{rem}[thm]{Remark}
\newtheorem{example}[thm]{Example}

\newtheorem{constr}[thm]{Construction}
\numberwithin{equation}{section}

\newcommand{\nosubsections}{\renewcommand{\thethm}{\thesection.\arabic{thm}}
                            \setcounter{thm}{0}
                           }

\newcommand{\cref}[3]{(\ref{#1}, #2 \ref{#3})}

\date{\today}

\usepackage{amssymb,amsmath}
\usepackage{multirow}

\usepackage{url}  

\setboolean{probleme}{false}
\setboolean{xlabels}{true}

\usepackage{graphicx}

\graphicspath{{Images/}} 

\theoremstyle{definition}

\DeclareMathOperator{\dx}{dx}
\DeclareMathOperator{\dy}{dy}

\begin{document}

\title
{New Solutions to the Poincar\'e Center Problem in Degree 3}

\author[von Bothmer]{Hans-Christian von Bothmer}
\address{Hans-Christian von Bothmer, Fachbereich Mathematik der Universit\"at Hamburg\\
Bundesstra\ss e 55\\
20146 Hamburg, Germany}
\email{hans.christian.v.bothmer@uni-hamburg.de}

\begin{abstract}
Let $\omega$ be a plane autonomous system and $C$ its configuration of algebraic integral curves. If the singularities of $C$ are quasi-homogeneous we we present new criteria that guarantee Darboux integrability. We use this to construct previously unknown components of the center variety.
\end{abstract}

\maketitle

\newcommand{\dual}{^*}
\newcommand{\barf}{\bar{f}}
\newcommand{\barg}{\bar{g}}
\newcommand{\barh}{\bar{h}}
\newcommand{\bara}{\bar{a}}
\newcommand{\barJ}{\bar{J}}
\newcommand{\barN}{\bar{N}}
\newcommand{\barH}{\bar{H}}

\newcommand{\impsi}{\overline{\Img \psi}}
\newcommand{\imphi}{\overline{\Img \phi}}
\newcommand{\semi}{Theorem \ref{tSemi}}

\newcommand{\ZZx}{\ZZ[x_1,\dots,x_n]}
\newcommand{\QQx}{\QQ[x_1,\dots,x_n]}
\newcommand{\CCx}{\CC[x_1,\dots,x_n]}
\newcommand{\FFp}{\FF_p}
\newcommand{\FFx}{\FFp[x_1,\dots,x_n]}
\renewcommand{\AA}{\mathbb{A}}

\newcommand{\zoladek}{\.Zo\l\c adek\,}
\newcommand{\zoladeks}{\.Zo\l\c adek's\,}

\section{Introduction}
\nosubsections
In 1885 Poincar\'e  \cite{Poincare1885}, asked under what conditions the differential equation
\[
y' = - \frac{x + p(x,y)}{y+q(x,y)} =: - \frac{P(x,y)}{Q(x,y)}
\]
with convergent power series $p(x,y)$ and $q(x,y)$ starting with quadratic terms, 
has stable solutions in a neighborhood of the equilibrium point
$(x,y)=(0,0)$. That is, all nearby solutions of the
corresponding planar autonomous system
\begin{align*}
	\dot{x} &= y + q(x,y) = Q(x,y)\\
	\dot{y} &= -x - p(x,y) = -P(x,y)
\end{align*}
are closed curves around the origin. Such a differential equation is said to have a 
{\sl center} at $(0,0)$.

Poincar\'e showed that one can iteratively find a formal power series
$F = x^2+y^2+f_3(x,y)+f_4(x,y)+\dots$ such that
\[
	\det \begin{pmatrix} F_x & F_y \\ P & Q \end{pmatrix} = \sum_{j=1}^\infty s_j(x^{2j+2}+y^{2j+2})
\]
with $s_j$ rational polynomials in the coefficients of $P$ and $Q$.
If all $s_j$ vanish, and $F$ is convergent then $F$ is a constant of motion, i.e. its gradient field
satisfies $Pdx+Qdy=0$. Since $F$ starts with $x^2+y^2$, this shows that close to the origin all integral curves are closed and the system is stable. Therefore the $s_j$'s are
called the {\sl focal values} of $Pdx+Qdy$. Often also the notation $\eta_{2j} := s_j$ is used, and the $\eta_i$ are called {\sl Lyapunov quantities}.

Poincar\'e also showed that the existence of an analytic constant of motion implies the vanishing of all focal values.
Later Frommer \cite{Frommer} proved that the systems above are stable if and only if all focal values vanish even without the assumption of convergence of $F$. (Frommer's proof contains a gap which can be closed \cite{vWahlGap}.)

Unfortunately, it is in generally impossible to check Poincar\'e's condition for a given differential equation because
there are infinitely many focal values. In the case where $P$ and $Q$ are polynomials of degree
at most $d$, the $s_j$ are polynomials in finitely many unknowns. Hilbert's Basis Theorem then implies
that the ideal $I_\infty = (s_1,s_2,\dots)$ is finitely generated and that the solution set is an algebraic variety, the {\sl center variety}.

Poincar\'e was inspired by work of Darboux \cite{Darboux} who showed that the existence of algebraic integral curves often implies the existence of a constant of motion. Such systems are now called {\sl Darboux integrable}. We review Darboux's Theorem in Section \ref{sPrelim}. 

For $d=2$ the center variety has $4$ components and each component is characterized by a certain type of algebraic integral curve configuration:

\begin{enumerate}
\item three lines in general position
\item a line and a conic in general position
\item infinitely many cubics
\item a conic and a cubic in special position.
\end{enumerate}

All these differential forms are Darboux integrable with respect to the given curve configuration. \cite{Dulac}

In degree $3$ a new type of center appears, the rational reversible centers. These were classified by \zoladek \cite{zoladekRational}. On the other hand the classification of Darboux centers in degree $3$ is still open and appears to be very difficult.
So for $d=3$ only partial results are known, for example  \cite{ZoladekRomanovskii} and \cite{ChristopherSubspace}. In \cite{zoladekCorrection} \zoladek gives a list of $52$ families of differential
forms known to have centers. $35$ of these are Darboux integrable. In \cite{survey} we analyze which of \zoladeks families form reduced components of the center variety (22 of the 52). The other ones either give non reduced components or are proper subfamilies of reduced components. 

Computer experiments over finite fields \cite{survey}, \cite{Ste:2011} seem to indicate that \zoladeks list is complete up to codimension $8$ and that, at higher codimension, more than 100 additional components exist.
It is the aim of this paper to present a geometric construction for some of these predicted components and give a rigorous proof of their existence over $\CC$.

Our construction is via curve configurations in special position with simple singularities. This construction consists of two parts:

Firstly, for a reduced, but possibly reducible and singular curve $C$ we consider the vector space of $V_d(C)$ differential forms of degree $d$ having the given curve configuration among their integral curves.  The computation of $V_d(C)$ is called the {\sl inverse problem} and was solved by Christopher, Llibre, Pantazi and Walcher in \cite{inverseProblemWalcherSuggestion}. In Section \ref{sInverse} we recall these results and work out the contribution of the singularities of $C$ to the expected dimension of $V_d(C)$. The case of nodes was already treated in \cite{inverseProblemWalcherSuggestion}. In general, it turns out that each singularity contributes via its {\sl Tjurina number} (with some modifications for points at infinity.) Furthermore, we investigate how $V_d(C)$ behaves in families. A key geometric insight from \cite{inverseProblemWalcherSuggestion}
allows us to describe this explicitly. 

Given a family of differential forms constructed from a family of curve configurations using the methods of Section \ref{sInverse}, our goal is to prove Darboux integrability for each differential form in the family.  For this we extend Darboux's method to include information about the singularities of the curve configurations. More precisely for a differential form $\omega$ with given integral curves $\{C_1,\dots,C_r\}$ we evaluate the cofactors of the integral curves $dK_{C_i}$ and $d\omega$ at the singular points of the configuration. These values were previously studied in \cite{necessaryChavarriga} to rule out the existence of integral curves in certain situations of interest. Here we prove that for simple singularities the ratio $\eta$ of these values depends solely on the type of the singularity {\sl and not on $\omega$}. Using $\eta$ we find special linear combinations of the cofactors and $d\omega$ for which we can apply Darboux' Theorem. These ideas are developed in Section \ref{sDarboux}.
They are the first main contribution of this paper. 

Finally, in Section \ref{sConstruction} we use the above methods to construct new components of the $d=3$ center variety in codimension $9$.  In each case we describe an equi-singular family of curve configurations and then use the singularities to solve the inverse problem and to prove Darboux integrability. Furthermore we compute the tangent space to the center variety in one explicit example and use this to prove that our families are components, i.e. cannot be extended. These new components are the second main contribution of this paper.

We found the special curve configurations used for these constructions by analyzing data obtained from computer experiments over a finite field \cite{survey}, \cite{Ste:2011}. 

The author would like to thank Jaume Llibre and Sebastian Walcher for reading and commenting on an earlier version of this paper, and in particular for pointing out the article \cite{inverseProblemWalcherSuggestion}, which greatly improved and simplified the presentation of the inverse problem in Section \ref{sInverse}.

The author would also like to remember Jakob Kr\"oker. Without his computational work and friendship, this project would have been impossible. 

\section{Preliminaries} \xlabel{sPrelim}
\nosubsections

In this article, we describe a plane autonomous system by a differential form $\omega = P\dx + Q\dy$ where $P, Q \in \CC[x,y]$ are polynomials of degree at most $d$.
\begin{defn}
Let $\omega$ be a differential form, and let $F, \mu \in \CC[[x,y]]$ be power series. If
\[
	dF = \mu \omega
\]
then $F$ is called a {\sl first integral} and $\mu$ an {\sl integrating factor} of $\omega$
\end{defn}

For a given $\omega$ it can often be very difficult to decide, whether a first
integral exists. Darboux realized in 1878 that the existence of algebraic
integral curves can help to answer this question: 

\begin{defn}
Let $C \in \CC[x,y]$ be a polynomial and $\{ C=0 \}$ the plane algebraic curve defined by $C$. The zero locus
$\{C=0\}$ is called an {\sl algebraic integral curve} of a differential form $\omega$ if and only if
\[
	dC \wedge \omega |_C = 0 \iff dC \wedge \omega = C\cdot K_C
\]
In this situation the $2$-form $K_C$ is called the {\sl cofactor} of $C$. In a slight abuse of notation, 
we  also denote the algebraic curve $\{C=0\}$ by the same letter $C$.
\end{defn}

\begin{thm}[Darboux 1878]
Let $\omega$ be a differential form, $C_1,\dots,C_r$ algebraic integral
curves of $\omega$ and $K_1,\dots,K_r$ their cofactors. 
\begin{itemize}
\item If $\sum \alpha_i K_i = -d\omega$ for appropriate $\alpha_i \in \CC$ then 
$\mu = \prod C_i^{\alpha_i}$ is a rational integrating factor of $\omega$.
\item If $\sum \alpha_i K_i = 0$ for appropriate $\alpha_i \in \CC$ then 
$F = \prod C_i^{\alpha_i}$ is a first integral of $\omega$.
\end{itemize}
\end{thm}

\begin{proof}
For the first claim we calculate:
\begin{align*}
	d(\mu \omega) 
	&= d\mu \wedge \omega + \mu d \omega \\
	& = \mu \bigl( \frac{d \mu}{\mu} \wedge \omega + d\omega \bigr) \\
	& = \mu ( d\log \mu \wedge \omega + d\omega) \\
	&= \mu \bigl( \sum \alpha_i d \log C_i \wedge \omega + d \omega \bigr) \\
	&= \mu \bigl( \sum \alpha_i \frac{dC_i}{C_i} \wedge \omega + d \omega \bigr) \\
	&= \mu \bigl( \sum \alpha_i K_i+ d \omega \bigr) \\
	&= 0
\end{align*}
consequently there exists an $F$ with $dF = \mu \omega$.

For the second claim we observe that 
$$
	dF = \mu \omega \iff dF \wedge \omega = 0.
$$
We now compute
\begin{align*}
	dF \wedge \omega 
	&=  F \Bigl(\frac{d F}{F} \wedge \omega \Bigr)\\
	&=  F (d \log F \wedge \omega)\\
	& = F \sum \alpha_i d\log C_i \wedge \omega \\
	& = F \sum \alpha_i K_i\\
	&= 0
\end{align*}
\end{proof}

Sometimes the following trivial observation is useful:

\begin{lem} \xlabel{lUnionOfIntegralCurves}
Let $C, D$ be plane algebraic curves without common components and $\omega$ a differential form
Then $C$ and $D$ are integral curves of $\omega$ if and only if $C\cup D$ is an integral curve of $\omega$.
\end{lem}

\begin{proof} The integral curve condition can be checked locally, outside the intersection
points of $C$ and $D$. Since it is a closed condition, it suffices to verify the condition on a dense open subset of each component.
\end{proof}

\section{Singularities of Plane Algebraic Curves} \xlabel{sSing}
\nosubsections

We now turn to singular algebraic integral curves. To this end, we recall
some singularity theory of plane algebraic curves, see \cite{singGreuel} for a detailed introduction.

\begin{defn} Let $F \in \CC[[x,y]]$ be the germ of a reduced curve singularity at $0$. Then
\[
	m = \dim \frac{\CC[[x,y]]}{(F_x,F_y)}
\]
is called the {\sl Milnor number} of $F$. Similarly
\[
	t = \dim \frac{\CC[[x,y]]}{(F_x,F_y,F)}
\]
is called the {\sl Tjurina number} of $F$. The Milnor number is a topological invariant of $C$, whereas the Tjurina number is only an analytic invariant. Note that by definition $m\ge t$, with equality holding for quasi-homogeneous singularities.
\end{defn}


\begin{example} \xlabel{eSimpleSingularities} The so-called {\sl simple singularities} are:

\begin{center}
\begin{tabular}{|c|c|c|c}
\hline
name & local equation & $m=t=c$ \\ 
\hline
node ($A_1$)		& $x^2-y^2$ & 1  \\
cusp  ($A_2$)		& $x^2-y^3$ & 2  \\
tacnode ($A_3$)	& $x^2-y^4$ & 3  \\
$A_n$ 			& $x^2-y^{n+1}$ & n  \\
triple point ($D_4$) & $x^3-y^3$ & 4  \\
$D_n$			& $y(x^2-y^{n-2})$ & n  \\
$E_6$			& $x^3-y^4$ & 6 \\
$E_7$			& $x(x^2-y^3)$ & 7  \\
$E_8$			& $x^3-y^5$ & 8  \\
\hline
\end{tabular}
\end{center}

The simple singularities are all quasi-homogeneous and therefore $m=t$. Furthermore,
for simple singularities, the maximum 
number of conditions $c$ a polynomial must satisfy in order to have a singularity of this type is also equal to the Milnor number.

In addition to the simple singularities, we also need
four-fold points which have a local equation $x^4+\alpha x^2y^2 + y^4$
with arbitrary $\alpha$. They have $m=t=9$.
\end{example}

We next consider certain special points at infinity. Here the
conditions appearing in the center focus setting do not 
match those in singularity theory:

\begin{lem} \xlabel{lTjurinaInfinity}
Let $F \in \CC[x,y,z]$ be a homogeneous polynomial of degree $e$.
Then 
\[
	\bigl( F', (F_x)', (F_y)' \bigr) = \bigl(F',(F')_y,z(F')_z \bigr) \subset \CC[y,z]
\]
where we denoted by a prime the dehomogenization of a homogeneous polynomial with respect to $x$. \end{lem}

\begin{proof}
Since $F$ is homogeneous of degree $e$ we have the Euler relation
\[
 eF = xF_x+yF_y + zF_z.
\] 
Dehomogenizing with respect to $x$, we obtain
\[
	dF' = (F_x)' + y(F_y)' + z(F_z)'
\]
and hence $(F_x)' = eF' - y(F_y)' + z(F_z)'$. This shows
\[
	\bigl( F', (F_x)', (F_y)' \bigr) =
	\bigl( F', z(F_z)', (F_y)' \bigr).
\]
Observing that $(F_y)' = (F')_y$ and $(F_z)' = (F')_z$ gives the claim.
\end{proof}

\begin{defn}
Let $F \in \CC[[y,z]]$ be the germ of a reduced curve singularity at $(y,z)=(0,0)$ with no component on $z=0$. Then we call
\[
	t_z := \dim \frac{\CC[[y,z]]}{(F_y,zF_z,F)}
\]
the {\sl modified Tjurina number}. It is invariant under analytic base changes that preserve the line $z=0$. 
\end{defn}

In the case of quasi-homogeneous singularities we can give a geometric interpretation of the modified Tjurina number:

\begin{prop} \xlabel{pModifiedGeometric}
Let $F \in \CC[[y,z]]$ be the germ of a reduced quasi-homogeneous singularity at $(y,z)=(0,0)$ with no components contained in the line  $\{z=0\}$. Then 
\[
	t_z = t + i -1
\]
where $t$ is the usual Tjurina number and $i$ is the intersection multiplicity of $F$ with $\{z=0\}$.
\end{prop}

\begin{proof}
Since $F$ is quasi-homogeneous we have the weighted Euler relation 
\[
	(\deg F)F = (\deg y)yF_y + (\deg z)zF_z
\]
and hence
\[
	(F_y,zF_z,F) = (F_y,zF_z).
\]
Since $F$ has no components on $\{z = 0\}$, the derivative $F_y$ is not divisible by $z$. Therefore
\[
	(F_y,zF_z) =  (F_y,z) \cup (F_y,F_z).
\]
By the weighted Euler relation we also have
\[
	(F,z) = (yF_y,z) = (y,z) \cup (F_y,z).
\]
Hence, we find that
\[
	t_z = \dim \frac{\CC[y,z]}{(F_y,z)} + t
\]
and 
\[
	i = 1 + \dim \frac{\CC[y,z]}{(F_y,z)}.
\]
The claim follows. 
\end{proof}

\begin{example} \xlabel{eTjuInfinity} We will be interested in the following cases:

\begin{center}
\begin{tabular}{|c|c|c|c|}
\hline
condition on $F$ & $t$ & $i$ & $t_z$  \\ 
\hline \hline
transversal to the line at $\infty$ & 0 & 1 & 0 \\ \hline
tangent to the line at $\infty$ & 0 & 2 & 1 \\ \hline
general node at $\infty$ 	    & 1 & 2 & 2 \\ \hline
node at $\infty$ with one branch & \multirow{2}{*}{$1$}& \multirow{2}{*}{$3$} & \multirow{2}{*}{$3$}  \\
tangent to the line at $\infty$  & &  &\\ \hline
general triple point at $\infty$		& $4$ & $3$  & 6 \\
\hline
\end{tabular}
\end{center}
Since all cases in the table involve reduced quasi-homogeneous singularities without components on the line at infinity, we can use 
Proposition \ref{pModifiedGeometric} to compute $t_z$.
\end{example}

\section{Singularities and the Inverse Problem} \xlabel{sInverse}
\nosubsections

In this section, we recall the solution to the inverse problem from \cite{inverseProblemWalcherSuggestion}. Since we are interested in differential forms of a fixed degree $d$, we work with homogeneous polynomials and state the definitions and results of \cite{inverseProblemWalcherSuggestion} in the graded setting. The proofs remain the same. 

\newcommand{\h}[1]{\overline{#1}}

\begin{defn}
If $P \in \CC[x,y]$ is a polynomial of degree $d$, we denote by
$\h{P} \in \CC[x,y,z]$ its homogenization. For differential forms $\omega = P\dx + Q\dy$
we define $\h{\omega} := \h{P} dx + \h{Q} dy$.
\end{defn}

\begin{rem}
Notice that the definitions of algebraic integral curves and cofactors do not
change if we homogenize everything. Similarly, Darboux's Theorem also holds
for homogenized polynomials and differential forms. 
\end{rem}

\begin{defn}
Let $C_1,\dots,C_r \in \CC[x,y,z]$ be homogeneous polynomials.
Then 
\[
	 \begin{pmatrix}
	C_{1x} & C_{1y} & C_1 &  &   \\
	\vdots & \vdots & & \ddots & \\
	C_{rx} & C_{ry} & & & C_r
	\end{pmatrix}
\]
is called the {\sl Darboux Matrix} of the configuration $C_1,\dots,C_r$.
\end{defn}

\begin{prop} \xlabel{pDarbouxMatrix}
Let $C_1,\dots,C_r$ be a configuration of plane curves. A differential form $\omega = P\dx + Q\dy$ has integral curves $C_i$ with cofactors $K_i \dx\dy$ if and only if
\[
	M \cdot (Q,-P, -K_1,\dots,-K_r)^t = 0
\]
where $M$ is the Darboux matrix of $C_1,\dots,C_r$.
\end{prop}

\begin{proof} 
The definition of integral curve gives
\begin{align*}
 	0 &= dC_i \wedge \omega - C_iK_i \dx\dy \\
	   &= (C_{ix}\dx + C_{iy}\dy) \wedge (P\dx + Q\dy) - C_iK_i \dx\dy \\
	   &= (C_{ix}Q - C_{iy}P-C_iK_i) \dx\dy.
\end{align*}
Writing these equations in matrix form gives the claimed identity.
\end{proof}

Since, by Lemma \ref{lUnionOfIntegralCurves}, a reduced (possibly reducible) curve $C = C_1\dots C_r$ is an integral curve of a differential form if and only if all its irreducible factors are, we now restrict to the case of one integral curve.

\begin{defn}
Let $C \in \CC[x,y,z]$ be a homogeneous polynomial with no multiple factors and $M_C = (C_x,C_y,C)$ its Darboux matrix. Let
\[
	\sV_C(d) := (\ker M_C)_d
\]
be the space of degree-$d$ differential forms that admit $C$ as an integral curve.

The vector $H_C := (C_y,-C_x,0)^t$ is always in the kernel of $M_C$ and represents the Hamiltonian vector field associated to $C$. Let 
\[
	\sV^H_C(d) := (H_C)_d
\]
be the {\sl vector space of trivial differential forms}. Its elements are of the form $FH_C$, where $F \in \CC[x,y,z]$ is homogeneous of degree $d-e+1$.
\end{defn}

\newcommand{\ioc}{(C_x, C_y) : C}
\newcommand{\iocbigl}{\bigl(\ioc\bigr)}

\begin{lem} \xlabel{lidealOfCofactors}
Let $C \in \CC[x,y,z]$ be a homogeneous polynomial with no multiple factors and no components at infinity, and $K$ a second homogeneous polynomial. Then $K$ is the cofactor of a differential form $\omega$ admitting $C$ as an integral curve if and only if
\[
	K \in \ioc.
\]
Furthermore $\ioc$ is saturated.
We call $\ioc$ the {\sl ideal of cofactors} of $C$. 
\end{lem}

\begin{proof}
This is just the definition of ideal quotients, i.e.
\[
	K \in \iocbigl \iff
	\exists A,B \in \CC[x,y,z] \colon AC_x + BC_y = K.
\]
Since $(C_x,C_y)$ is a complete intersection it is saturated, and therefore $\ioc$ is also saturated.
\end{proof}

With this we have a first computation of the dimension of $\sV_C(d)$:

\begin{prop} \xlabel{pFirstDimensionFormula}
Let $C \in \CC[x,y,z]$ be a homogeneous polynomial with no multiple factors and no components at infinity. Then
\[
	\dim \sV_C(d) = \dim \sV_C^H(d) + \dim \iocbigl_{d-1}
\]
\end{prop}

\begin{proof}
Consider the linear map
\[
	\pi \colon \sV_C(d) \to \CC[x,y,z]
\]
sending a differential form $\omega$ admitting $C$ as an integral curve to its cofactor. The image of $\pi$ is the degree $d-1$ part of the ideal of cofactors. The kernel of $\pi$ consists of differential forms $P\dx+Q\dy$ with cofactor $0$. Such differential forms satisfy
\[
	QC_x - PC_y  = 0 \iff QC_x = PC_y.
\]
Since $C$ has no multiple factors, $C_x$ and $C_y$ have no nontrivial common factor. Therefore all solutions of the above equation are of the form
\[
	Q = FC_y, \quad P = FC_x
\]
for some homogeneous polynomial $F \in \CC[x,y,z]$. But these are just the elements of $\sV^H_C(d)$. 
\end{proof}

We now turn to a key geometric observation of \cite{inverseProblemWalcherSuggestion}, namely that the vanishing locus of the ideal of cofactors has a geometric meaning:

\begin{prop} \xlabel{pLinkage}
Let $C \in \CC[x,y,z]$ be a homogeneous polynomial of degree $e$ with no multiple factors and no components at infinity, and $M_C = (C_x,C_y,C)$ the Darboux matrix. Consider
\begin{enumerate}
\item $X$ the finite scheme defined by the vanishing of $M_C$,
\item $Y$ the finite scheme defined by the ideal of cofactors,
\item $W$ be the finite scheme defined by $C_x=C_y=0$ i.e. the scheme of critical points of the Hamiltonian vector field of $C$.
\end{enumerate}
Then
\[
	X \cup Y = W \quad \text{and} \quad X \cap Y = \emptyset.
\]
In particular
\[
	\deg Y = (e-1)^2 - \deg X.
\]
\end{prop}

\begin{proof}
Since $C$ has no multiple factors and no components at infinity, $C_x$ and $C_y$ have no common factors and $W$ is what is called a complete intersection in algebraic geometry. The relations
\[
	X \cup Y = W \quad \text{and} \quad X \cap Y = \emptyset
\]
are a standard fact of algebra \cite[p.\,291]{inverseProblemWalcherSuggestion}.
\end{proof}

\begin{rem}
The relation above is called {\sl linkage} in algebraic geometry, and $Y$ is referred to as {\sl  the set of points linked to $X$ via $W,$} see \cite[Section 21.10]{Ei95} for an introduction. Linkage can also be used to infer further properties of $Y$, beyond its degree, from those of $X$. This is not used in this paper, but might be a fruitful direction of further research.
\end{rem}

With Proposition \ref{pLinkage} we obtain a refined count of differential forms with a given integral curve:

\newcommand{\expectedDim}{\delta}

\begin{prop}  \xlabel{pExpected}
Let $C \in \CC[x,y,z]$ be a homogeneous polynomial of degree $e$ with no multiple factors and no components at infinity, let $M_C = (C_x,C_y,C)$ be the Darboux matrix, and $X \subset \PP^2$ the scheme defined by the vanishing of $M_C$. Then
\[
	\dim \sV_C(d) \ge {d-e+3 \choose 2} + {d+1 \choose 2} - (e-1)^2 + \deg X =: \expectedDim
\]
where the first summand is $0$ if $e > d+1$.
We call $\expectedDim$ the {\sl expected dimension} of differential forms with integral curve $C$.
\end{prop}

\begin{proof}
Firstly, we notice that $H_C$ is a differential form of degree $e-1$ and that every element of 
$(H_C)_{d}$ is of the form $aH_C$ with $a \in \CC[x,y,z]_{d-e+1}$. So
\[
	\dim V^H_C(d) = {d-e+3 \choose 2}.
\]
Secondly, since the ideal of cofactors is saturated, 
\[
	\iocbigl_{d-1}
\]
is the vector space of polynomials of degree $d-1$ vanishing on $Y$. This imposes at most $\deg Y$ conditions so
\[
	\dim \iocbigl_{d-1} \ge {d+1 \choose 2} - \deg Y.
\]
Using Proposition \ref{pFirstDimensionFormula} and Proposition \ref{pLinkage} we obtain the claimed inequality. 
\end{proof}

As a last step we compute $\deg X$ from the geometry of $C$.

\begin{prop} \xlabel{pTjurinaDegX}
Let $C \in \CC[x,y,z]$ be a homogeneous polynomial of degree $e$ with no multiple factors and no components at infinity, $M_C = (C_x,C_y,C)$ the Darboux matrix, and $X$ the finite scheme defined by the vanishing of $M_C$. Let $C_{\sing}$ the set of singular points of $C$ outside of the line at infinity, and $C_\infty$ the set of points of $C$ that lie on the line at infinity. Then
\[
	\deg X = \sum_{P \in C_{\sing}} t(P) + \sum_{P \in C_\infty} t_z(P)
\]
where $t(P)$, $t_z(P)$ denote the Tjurina and modified Tjurina number of the curve $C$ at $P$. 
\end{prop}

\begin{proof}
Since $X$ is finite we can compute the degree locally at the points. Outside of $z=0$ the solutions of $C_x=C_y=C=0$ are precisely the singular points of $C$. The multiplicity of each point is Tjurina number (by definition). On the line at infinity every point of $C$ can possibly occur in $X$. Lemma \ref{lTjurinaInfinity} shows that the multiplicity of these points in $X$ is given by the modified Tjurina number. 
\end{proof}

\begin{rem}
One cannot avoid the inequality in Proposition \ref{pExpected}. Let for example $C$ be an irreducible quartic curve with $3$ nodes and $C'$ the union of a line and a cubic curve. Aussume that both curves intersect the line at infinity transversally. Then both curves have degree $e=e' =4$ and $\deg X = \deg X' = 3$ and therefore $\deg Y = \deg Y'= 3^2-3=6$. One can compute that $Y$ is a set of $6$ points that do not lie on a conic, but $Y'$ is a special set  $6$ points that do lie on a conic. Since also $\dim \sV^H_C(3) = \dim \sV^H_{C'}(3) = 1$ in both cases, we obtain
\[
	\dim \sV_C(3) = 1 = \expectedDim
\]
while
\[
	\dim \sV_{C'}(3) = 2 > \expectedDim.
\]
 \end{rem}
 
 We now want to describe how $\sV_C(d)$ behaves in families. 
 
 \begin{defn}
 An irreducible variety $B \subset \PP\bigl(\CC[x,y,z]_e\bigr)$ is called an {\sl equi-singular family of plane curves of degree $e$ with modified Tjurina number $t$} if there exist a Zariski open subset $U \subset B$ such that for all $C \in U$  the scheme $X$ defined by $C_x=C_y=C=0$ is finite and satisfies $\deg X = t$.
\end{defn}

\begin{prop}  \xlabel{pFamilyGeneric}
Let $B$ be an {\sl equi-singular family of plane curves of degree $e$ with modified Tjurina number $t$}, $C \in B$ a curve and $\omega \in \sV_C(d)$ a differential form 
satisfying the following properties
\begin{enumerate}
 	\item $C$ is reduced, i.e. has no multiple factors,
	\item $C$ has no components at infinity,
	\item $\omega$ has only finitely many integral curves of degree $e$, \label{iFinite}
	\item $\deg X = t$, \label{iDegt}
	\item $\dim \sV_C(d) = \expectedDim$ as expected. \label{iEta}
\end{enumerate}
Let $\Omega_B$ be the  the family of differential forms (up to scaling) of degree $d$ that admit at least one of the algebraic integral curves $C \in B$. Let $\Omega_{B,\omega}$ be the component of $\Omega_B$ that contains $\omega$. Then
\[
	\dim \Omega_{B,\omega} = \dim B + \expectedDim - 1.
\]
\end{prop}

\begin{proof}
The first three conditions are open, and thus hold on a Zariski open neighbourhood of $C$ in $B$. Condition (\ref{iDegt}) is open on $B$ by definition of an equi-singular family. Now consider the ideal of cofactors. Firstly, for all $C'$ in the open neighbourhood of $C$ where the first $4$ conditions are satisfied we have 
\[
	\dim \bigl(( C'_x, C'_y ) : C'\bigr)_{d-1} \ge {d+1 \choose 2} - \deg Y'
\]
with $\deg Y = (e-1)^2 - \deg X = (e-1)^2-t$ constant. Secondly by condition (\ref{iEta})
\[
	\dim \bigl(\ioc\bigr)_{d-1} = {d+1 \choose 2} - \deg Y
\]
at the special point $C$. 

It is a well-known fact in algebraic geometry that the dimension of the vector space of polynomials vanishing on a finite scheme is semi-continuous in families, i.e. the strata where the dimension increases are closed. Since by the above conditions the dimension can not decrease, the strata where 
\[
	\dim \sV_C(d) \not= \expectedDim
\]
must be closed and the stratum with $\dim \sV_C(d) = \expectedDim$ containing $C$ is open. 
Taken together we have a Zariski open subset $U \subset B$ where $\dim \sV_C(d)$ is constant. We therefore have a vector bundle $\sV$ of rank $\delta$ over $U$ whose fiber over any $C' \in U$ is $\sV_{C'}(d)$. The projectivization $\PP(\sV)$ has dimension $\dim B + \expectedDim -1$.

Now consider the morphism from $\PP(\sV)$ to the projectivized space of degree $d$ differential forms. By (\ref{iFinite}) this map is finite, and its image also has dimension $\dim B + \expectedDim -1$. Since this image is contained in $\Omega_{B,\omega}$ this proves the claim.
\end{proof}

\begin{example}
Consider the family $B \subset \CC[x,y,z]_4$ of plane $3$-cuspidal quartics that are tangent to the line
at infinity (see Figure \ref{f910}). We estimate its dimension:

The space of all quartic curves is $\PP^{14}$. The condition of having a 
cusp has codimension $2$ in this space, so we have at least an $14-3\cdot2 = 8$ 
dimensional family of $3$-cuspidal quartics. The condition of being tangent to the
line at infinity is codimension $1$, so $B$ is at least $7$-dimensional. 

Let us now calculate the degree of $X$ for a specific element $C \in B$. We assume that $X$ is finite for the moment.  Then the degree of $X$ can be calculated locally at the special points. By Example \ref{eSimpleSingularities} this degree is $2$ at cusps and by Example \ref{eTjuInfinity} it is $1$ at
the tangent to infinity.

Putting this together, we see that $\deg X \ge 3\cdot 2 + 1=7$. The expected dimension $\expectedDim$ of
degree $2$ differential forms with this type of integral curve is
\[
	\expectedDim = {d-e+3 \choose 2} + {d+1 \choose 2} - (e-1)^2 + \deg X
	= 0 + 3 - 3^2+7 = 1
\]
It remains to verify that one example of such a configuration satisfies the conditions
of Proposition \ref{pFamilyGeneric}. This is the case, for example, for
\[
	C =  x^2y^2-2x^2yw-2xy^2w+x^2w^2-2xyw^2+y^2w^2
\]
with $w =  (1/8)(x+y-z)$.

This proves that there exists a $7$-dimensional family of degree $2$ differential forms whose algebraic integral curves include a 3-cuspidal quartic tangent to the line at infinity.
\end{example}

\section{Singularities and Darboux Integrability} \xlabel{sDarboux}
\nosubsections

In this section we introduce an analytic invariant $\eta$ attached to a triple $(\omega,C,a)$ where $\omega$ is a differential form in the plane, $C$ is an algebraic integral curve of $\omega$, and $a$ is any point in the plane. While $\eta$ is defined for all points, it carries the most information when $a$ is a zero of $\omega$.

The main result of this section will be, that if $C$ has a quasi-homogeneous singularity at $a$, the value of $\eta$ will not depend on $\omega$. This will allow us to prove Darboux integrability for differential forms constructed via the methods of the previous section from a plane curve $C$, solely based on the singularities of $C$.

\newcommand{\pfrac}[2]{( #1 : #2)}
\newcommand{\bigpfrac}[2]{\bigl( #1 : #2\bigr)}
\newcommand{\sfrac}[2]{\frac{#2}{#1}}

\begin{defn}
Let $\omega$ be a differential form with algebraic integral curve $C$ and cofactor $K$ and $a$ a zero of $\omega$. Then we define
\[
	\eta := \eta(\omega,C,a) := \bigpfrac{K(a)}{d\omega(a)},
\]
where we interpret the right-hand side as a ratio, i.e $(u:v) = (\lambda u : \lambda v)$ for $\lambda \not =0$. The degenerate case $(0:0)$ is also allowed. 
\end{defn}

\begin{rem}
The values $K(a)$ at a zero of $\omega$ have also been considered in \cite{necessaryChavarriga} to find necessary conditions for the existence of algebraic integral curves. 

Here, we use $\eta$ to obtain sufficient conditions for Darboux integrability of $\omega$ with respect to given integral curves. 
\end{rem}

\begin{prop} \xlabel{pSingularPoints}
Let $\omega$ be a differential form with algebraic integral curve $C \in \CC[x,y]$. If $C$ has a quasi-homogeneous singularity at $(0,0)$ and has no multiple factors, then either 
\[
	\eta\bigl(\omega,C,(0,0)\bigr) =  \pfrac{\deg{C}}
	{\deg{x}+\deg{y}}
\]
or $\eta = (0:0)$.
\end{prop}

\begin{proof}
Since $C$ is quasi-homogeneous, we have the weighted Euler relation
\[
	(\deg C)C = (\deg x)xC_x + (\deg y)yC_y
\]
The Darboux matrix of $C$ is
\[
	M = (C_x,C_y,C)
\]
and its syzygy matrix is given by
\[
\begin{pmatrix}
(\deg x)x & C_y \\
(\deg y)y & -C_x \\
-\deg C & 0
\end{pmatrix}
\]
The first column comes from the Euler relation. Since $\deg C \ne 0$ is a constant, all further syzygies can be assumed to have a zero in the last row. Since $C$ has no multiple factors $C_x$ and $C_y$ have no common factor. Therefore all syzygies between these are generated by the one in the second column.

The two columns represent the following differential forms and cofactors:
\begin{align*}
\omega_1 &= (\deg x)x \dy -  (\deg y)y \dx  & K_1&=\deg C \dx \wedge \dy \\
\omega_2 & = C_y \dy +  C_x \dx & K_2 &=0
\end{align*}
We compute
\begin{align*}
	d\omega_1 &= (\deg x + \deg y) \dx \wedge \dy \\
	d\omega_2 &= (C_{yx} - C_{xy})\dx \wedge \dy = 0.
\end{align*}
Observe also that $\omega_1(0,0) = 0$ and also $\omega_2(0,0) = 0$ since $C$ is singular at $0$.

\medskip

Since $\omega_1$ and $\omega_2$ generate the vector space of all differentials with this integral curve, we must have
\[
	\omega = F \omega_1 + G\omega_2
\]
We evaluate $d\omega$ at $(0,0)$:
\begin{align*}
	d \omega (0,0) 
	&= (dF\omega_1 + Fd\omega_1 + dG\omega_2 + Gd\omega_2)(0,0) \\
	&= (\deg x + \deg y)F(0,0).
\end{align*}	
Since the cofactor of $\omega$ is $K = FK_1 + GK_2$ we have
\[
	K(0,0)= (\deg C)F(0,0).
\]
This completes the proof.
\end{proof}

\begin{example}
Using Proposition \ref{pSingularPoints} we can compute $\eta$ for the simple singularities, see Table \ref{tSimple}. It was pointed out to us by Michel L\"onne that these ratios agree with the log canonical thresholds of the simple singularities. It is not clear to us, why this is the case. 

\begin{table}
\begin{tabular}{|c|c|c|c|c|c}
\hline
Type & Equation & $\eta$, if not (0:0) \\
\hline
node ($A_1$)		& $x^2-y^2$ & 1:1  \\
cusp  ($A_2$)		& $x^2-y^3$ & 6:5  \\
tacnode ($A_3$)	& $x^2-y^4$ & 4:3 \\
$A_n$ 			& $x^2-y^{n+1}$ & (2n+2):n+3  \\
triple point ($D_4$) & $x^3-y^3$ & 3:2  \\
$D_n$			& $y(x^2-y^{n-2})$ & (2n-2):n \\
$E_6$			& $x^3-y^4$ & 12:7 \\
$E_7$			& $x(x^2-y^3)$ & 9:5   \\
$E_8$			& $x^3-y^5$ & 15:8  \\
\hline
\end{tabular}

\caption{$\eta$ for simple singularities} 
\label{tSimple}
\end{table}
\end{example}

Let us now see how this information allows us to prove Darboux integrability:

\begin{example}
Consider $\omega$ of degree $2$ with a $3$-cuspidal quartic $C$ as integral curve (see Figure \ref{f910}). Evaluating
at the three cusps $c_1$, $c_2$ and $c_3$ we obtain
\[
	\bigl(K(c_i) : d\omega(c_i) \bigr) = (6:5) \,\text{or}\, (0:0) 
\]
for all $i$. In any case $5K-6d\omega$ vanishes at $c_1, c_2$ and $c_3$. Now, $\omega$ is of degree $2$ and therefore $K$ has degree $1$. A linear form that vanishes at $3$ non-collinear points
must vanish identically. By Darboux's Theorem this implies that $\omega$ has a Darboux integrating factor.
\end{example}

We can also define our invariant for several integral curves at the same time:

\begin{defn}
Let $\omega$ be a differential form with algebraic integral curves $C_1,\dots,C_r$, and cofactors $K_1,\dots,K_r$ and $a$ a zero of $\omega$. Then we define
\[
	\eta := \eta(\omega,\{C_1,\dots,C_r\},a) := \bigl(K_1(a) : \dots : K_r(a) : d\omega(a) \bigr)
\]
where we again interpret the right-hand side as a ratio and the degenerate case $(0:\dots:0)$ is also allowed. 
\end{defn}

We compute this invariant for two intersecting integral curves:

\begin{prop} \xlabel{pTwoCurves}
Let $\omega$ be a differential form with algebraic integral curve $C,D \in \CC[x,y]$ intersecting at $(0,0)$. If $C\cdot D$ has no multiple factors and has a quasi-homogeneous singularity at $(0,0)$ and
\[
	\deg CD > \deg x + \deg y
\]
then	
\[
	\eta\bigl(\omega,
	\{C,D\},(0,0)\bigr) =  (\deg{C}:\deg D:\deg x + \deg y)
	\quad \text{ or } \quad (0:0:0).
\]
\end{prop}

\begin{proof}
The Darboux matrix of $C,D$ is
\[
	M = \begin{pmatrix}
	 C_x & C_y & C & \\
	 D_x & D_y &    & D \\
	\end{pmatrix}
\]
We claim that the kernel of $M$ is presented by
\[
\begin{pmatrix}
	(\deg x)x & -CD_y \\
	(\deg y)y & CD_x \\
	-\deg C   & C_xD_y - C_yD_x \\
	-\deg D   & 0 
\end{pmatrix}.
\]
The first column comes from the weighted Euler relations for $C$ and $D$. 
The entries of the second column are the $2 \times 2$ minors of the first three columns of $M$ (with appropriate signs). Both are in the kernel. We need to prove that they generate. 
\medskip 

Since $\deg D$ is a nonzero constant we can restrict to the case of elements in the kernel that have the last entry equal to zero. Each such element is in the kernel of the matrix $M_{123}$ obtained from the first $3$ columns of $M$. It is a well know fact from commutative algebra that the kernel of a $2 \times 3$ matrix is generated by the vector of $2 \times 2$ minors if the common zero locus of these minors is of codimension $2$. Here $CD_y=CD_x=0$
if and only if  $D_x=D_y=0$ or $C=0$. In the first case the zero locus consists of points, since $D$ has no multiple components. In the second case we have $C=0$ and $C_xDy - C_yD_x=0$. If this has a $1$-dimensional solution it is a component of $C$ that is an integral curve of the Hamiltonian differential form of $D$, but the only ones of these that also pass through the origin are the components of $D$. Since $C$ and $D$ have no common components the locus $C = C_xDy - C_yD_x = 0$ is also finite. 
\medskip

Let $\omega_1$ and $\omega_2$ be the differential forms represented by the first and second columns, and $K_{C,i}, K_{D,i}$ their cofactors. Notice that by the assumption on the degrees $\omega_2, d\omega_2, K_{C,2}$ are all quasi-homogeneous of positive degree. Therefore the all vanish at $(0,0)$. The same computation as in the proof of 
Proposition \ref{pSingularPoints} gives the claim.
\end{proof}

\begin{example}
Using Proposition \ref{pSingularPoints}, we compute $\eta$ in the case where $CD$ has a simple singularity at the intersection point of $C$ and $D$. Also we compute $\eta$ in an ordenary $4$-fold point, where one branch belongs to $C$ and three branches belong to $D$. See Table \ref{tTwoCurves}. 

\begin{table}
\begin{tabular}{|c|c|c|c|c|c}
\hline
Type & $C\cdot D$ & $\eta$\\
\hline
node ($A_1$)		& $(x-y)(x+y)$ & no information  \\
tacnode ($A_3$)	& $(x-y^2)(x+y^2)$ & 2:2:3 \\
$A_{2n-1}$ 		& $(x-y^n)(x+y^n)$ & n:n:(n+1)  \\
triple point ($D_4$) & $y(x^2-y^2)$ & 1:2:2  \\
$D_n$			& $y(x^2-y^{n-2})$ & 2:2(n-2):n \\
$E_7$			& $x(x^2-y^3)$ & 3:6:5   \\
\hline
$4$-fold point		& $x(x^3-y^3)$ & 1:3:2   \\
\hline
\end{tabular}

\caption{$\eta \not= (0:0:0)$ for simple singularities with at least two branches, and for a fourfold point.} 
\label{tTwoCurves}
\end{table}
\end{example}

\begin{example} \xlabel{eThreeCuspidalAndLine}
Let $C$ be a $3$-cuspidal cubic and $L$ a tangent line to the cubic in a smooth point $B$. Denote the cusp points by $R$, $S$ and $T$ (see Figure \ref{f98}). Let $\omega$ be differential form 
with $C$ and $L$ as integral curves and let $K_C$ and $K_L$ denote the corresponding cofactors. 
Then the evaluation of $(K_L:K_C:d \omega)$ at $B$ and the cusps gives the following values
\[
	\begin{pmatrix}
		 2 & 2  & 3\\
		 0 & 6  & 5 \\
		 0 & 6  & 5\\
		 0 & 6  & 5
	\end{pmatrix}
\]
(or zero). Since the kernel of this matrix is generated by $(4,5,-6)^t$, the linear combination
\[	
	 4 K_L + 5 K_C - 6 d\omega
\]
vanishes on $B$ and at all cusps. If the case $(0:0:0)$ does not occur for any of the points, this
is also the only possible linear combination that could possibly be zero. 
\end{example}

\newcommand{\lineinfty}{\PP^1_{\infty}}

We now consider the line at infinity. The first case is interesting when there are more integral curves than intersection points at infinity. Although this is not used in the present paper, we record it here for reference since it is easy to prove.

\begin{prop} \xlabel{pInfinityLineZero}
Let $\omega$ be a differential form of degree $d$ and $C_1, \dots, C_r$ integral
curves of $\omega$ with cofactors $K_1,\dots,K_r$. If $\lineinfty$ is the line at infinity
and $\{p_1,\dots,p_k\} \subset \lineinfty$ the set of points lying on at least one of the integral curves. Counting with multiplicities, one can write
\[
		C_i \cap \lineinfty = \sum_j {\alpha_{ij}}p_j
\]
for positive integers $\alpha_{ij}$. Let $A = (\alpha_{ij})$ be the matrix of coefficients
and $b = (\beta_1,\dots,\beta_k)^t$ a vector such that $A\cdot b = 0$. Then
\[
	\Bigl(\sum \beta_iK_i\Bigr)|_{z=0} = 0.
\]
\end{prop}

\begin{proof}
If $M$ is the Darboux matrix of $C_1, \dots, C_r$ and $s = (Q,-P,-K_1,\dots,-K_r)$ is the syzygy corresponding to $\omega$ and the cofactors, then
we can restrict the equation
\[	
	M \cdot s = 0
\]
to $\lineinfty$ by setting $z=0$. We then have 
\[
	C_i|_{z=0} = \prod_j L_i^{\alpha_{ij}}
\]
for linear equations $L_j$ with $\{L_j=0\} \cap \lineinfty = p_j$. Observe now, that
$M|_{z=0}$ is the same as the Darboux matrix of $C_1|_{z=0},\dots,C_r|_{z=0}$ and
that $\omega|_{z=0}$ is a differential form with integral curves $C_i|_{z=0}$ and
Cofactors $K_i|_{z=0}$.

It follows from  Lemma \ref{lUnionOfIntegralCurves} that $L_1,\dots, L_k$ define algebraic integral curves
of $\omega|_{z=0}$. Let $M_i$ be the corresponding cofactors. Now
\[
	C_i|_{z=0} = \prod_j L_i^{\alpha_{ij}}
\]
implies 
\[
	K_i|_{z=0} = \sum_j \alpha_{ij} M_j
\]
In particular $(\sum \beta_i K_i)|_{z=0} = 0$ if
\[
	A \cdot b = 0
\]
\end{proof}

Another interesting situation arises when the number of intersection points at infinity is large compared to the degree of the differential form. This is used in Construction \ref{c98}.

\begin{prop} \xlabel{pManyPointsInfinity}
Let $\omega$ be a differential form of degree $d$ and $C_1, \dots, C_r$ integral
curves of $\omega$ with cofactors $K_1,\dots,K_r$. If $\lineinfty$ is the line at infinity
and $\{p_1,\dots,p_k\} \subset \lineinfty$ the set of points that also lie on at least one integral curve. If $k > d+1$ and $x \in \lineinfty$ is a point, then
either: 
\[
	\eta = \eta\bigl(\omega,\{C_1,\dots,C_r\},x\bigr) = (\deg C_1: \dots : \deg C_i : d+1)
\]
or 
\[
\eta = (0:\dots:0).
\]
\end{prop}

\begin{proof}
As in the  proof of Proposition \ref{pInfinityLineZero} we have homogeneous linear forms $L_i$ for every point
$P_i$ and write 
\[
	C_i|_{z=0} = \prod_j L_i^{\alpha_{ij}}
\]
the differential form $\omega|_{z=0}$ has, on the one hand, integral curves $C_i|_{z=0}$ with  cofactors $K_i|_{z=0}$, but on the other hand integral curves $L_i$ with cofactors $M_i$.
As before this implies
\[
	K_i|_{z=0} = \sum_j \alpha_{ij} M_j.
\]
We now analyze the $M_j$ in more detail. Consider the Darboux matrix of the
$L_i$:
\[
	N =  \begin{pmatrix}
	L_{1x} & L_{1y} & L_1 &  &   \\
	\vdots & \vdots & & \ddots & \\
	L_{kx} & L_{ky} & & & L_k
	\end{pmatrix}
\]
Notice that $L_{ix}, L_{iy} \in \CC$. Therefore the kernel of  $N$ is generated by
\[
	\begin{pmatrix}
		x & -L_{ky}\Lambda \\
		y & L_{kx}\Lambda \\
		-1 & \left| \begin{smallmatrix} L_{1x} & L_{1y} \\ L_{kx} & L_{ky} \end{smallmatrix} \right|
		        \frac{\Lambda}{L_1}\\
		\vdots & \vdots \\
		-1 & \left| \begin{smallmatrix} 
		                    L_{k-1,x} & L_{k-1,y} \\ L_{kx} & L_{ky} 
		                 \end{smallmatrix} \right|
		        \frac{\Lambda}{L_{k-1}}\\
		-1& 0 \\
	\end{pmatrix}
\]
where $\Lambda = \prod_{i=1}^{k-1}L_i$. The first syzygy is the Euler relation, the second consists of the $r \times r$ minors not involving the last column. 

Now $(Q,-P,-M_1,\dots,-M_k)|_{z=0}$ must be a combination of the above columns. Since
in the second column 
$$\deg L_{ry} \Lambda = k-1 > d = \deg P$$
this column cannot be involved in the linear combination. It follows that
\[
	\omega|_{z=0} = Fx\dy -Fy\dx \quad \text{and} \quad M_i = F
\]
for an appropriate $F$ of degree $d-1$. Substituting this into the formula of $K_i|_{z=0}$
we find:
\[
	K_i|_{z=0} = \Bigl(\sum_j \alpha_{ij}\Bigr) F = (\deg C_i)F.
\]
Furthermore we have
\begin{align*}
	d(\omega) 
	&= d\bigl((Fx\dy -Fy\dx)\bigr)\\
	&= dF\wedge(x\dy-y\dx) + Fd(x\dy-y\dx) \\
	&= (F_x\dx + F_y\dy) \wedge(x\dy-y\dx) + 2F\dx\dy \\
	&=(F_xx + F_yy-2F)\dx\dy \\
	&=(\deg F+2)F \dx\dy \\
	&=(d+1)F \dx\dy
\end{align*}
It follows that for any point $x \in \lineinfty$ either $d\omega(x) = K_1(x) = \dots K_r(x) = 0$, or
\[
	\eta = (K_1(x) : \dots : K_r(x) : d\omega(x)) = (\deg C_1: \dots : \deg C_i : d+1)
\]
\end{proof}

\begin{example} \xlabel{eThreeCuspidalAndLine2}
Consider again a 3-cuspidal quartic curve $C$ and a tangent line $L$ as in Example \ref{eThreeCuspidalAndLine} and in Figure \ref{f98}. In this case, there are $5$ distinct intersection points at infinity. If, in addition, we assume that  $\omega$ has degree $3$ we can apply  Proposition \ref{pManyPointsInfinity} and for a general point $x \in \lineinfty$ we obtain
\[
	(K_L:K_C:d \omega) = (1:4:4).
\]
Together with the values at the cusps and at the tangent point, we obtain the following
matrix of values
\[
	\begin{pmatrix}
	   0 & 6 & 5 \\
	   2 & 2 & 3 \\
	   1 & 4 & 4
         \end{pmatrix}
\]
Observe that $(4,5,-6)^t$ still lies in the kernel of this matrix.. This implies that
\[	
	Q =  4 K_L - 5 K_C - 6 d\omega
\]
vanishes on $\lineinfty$ and at $4$ points outside this line. Since $Q$ is a
degree $2$ polynomial and the $4$ points do not lie on a line, it follows that $Q$ must vanish identically. 
This implies that $\omega$ has a Darboux integrating factor.
\end{example}

\section{Constructions} \xlabel{sConstruction}
\nosubsections

In this section, we construct Darboux-integrable differential forms of degree $d=3$ by using the methods of the previous sections. All of these constructions were found by analyzing finite-field examples obtained using the methods described in \cite{survey} and \cite{Ste:2011}. For computations we use \verb#Macaulay2# \cite{M2} and the scripts at 
\cite{cfGitHub}.

\begin{figure}[h]
\includegraphics*[width=10cm]{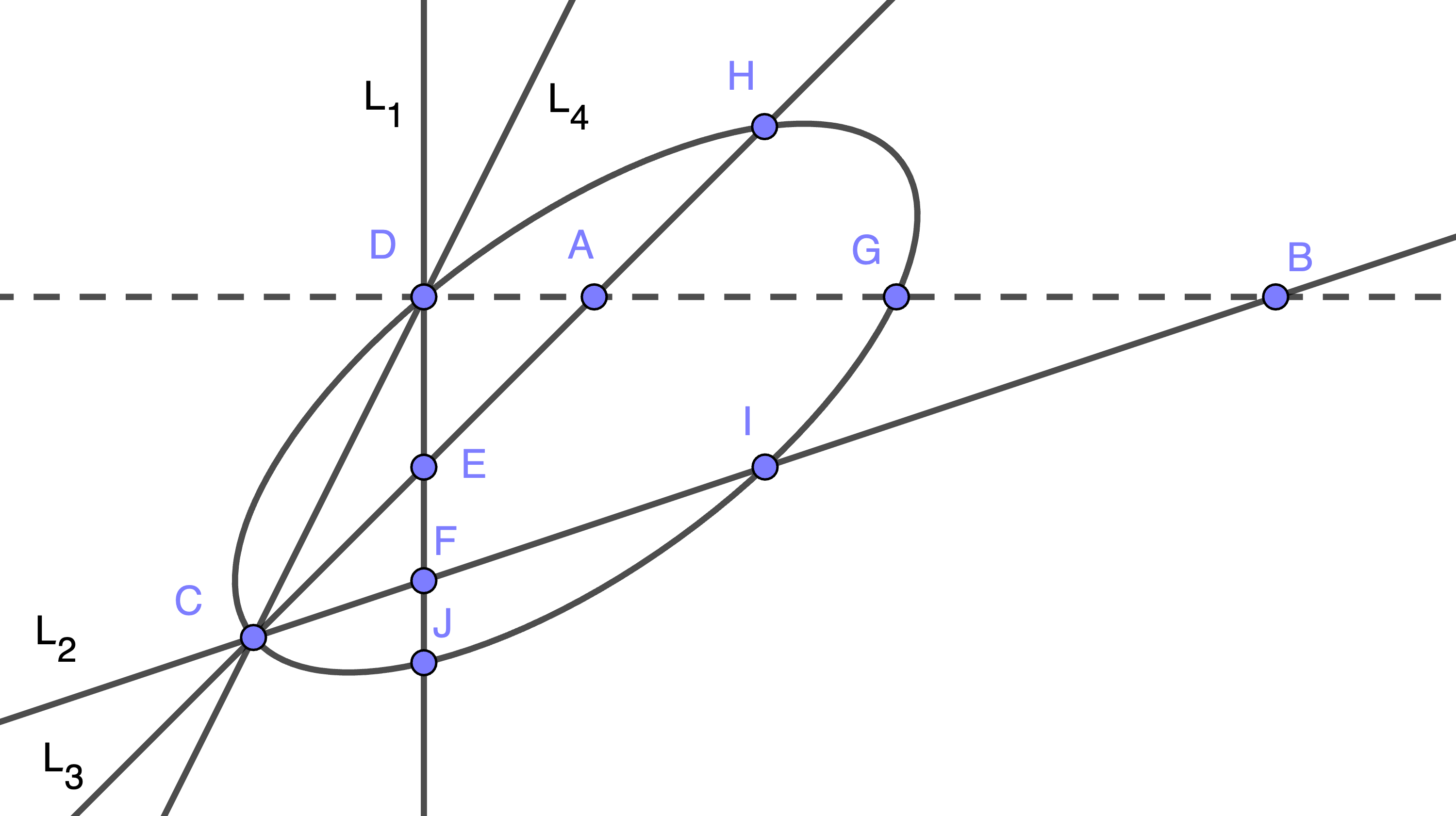}
\caption{Curve configuration for Construction \ref{c96}}
\label{f96}
\end{figure}

\begin{constr} \xlabel{c96}
Let $L_4$ be a line in $\PP^2$, $D$ the point where it intersects $\lineinfty$ and $C$ another
point on $L_4$. Now choose $L_1$ and $L_2$ through $C$ and $L_3$ through $D$. Finally choose a conic $Q$ through $C$ and $D$.  This gives a $2+0+1+1+1+1+3=9$ dimensional family of curve configurations. The union $U$ of all these curves has degree $e = 6$. Outside of $C$ and $D$ it has $5$ more nodes labeled $E, F, H, I, J$ in Figure \ref{f96}. 

The degree $t$ of $X = V(U,U_x,U_y)$ at the special points is:
\begin{center}
\begin{tabular}{|c|c|}
\hline
point & t \\
\hline
$4$-fold point at $C$ & 9 \\
$3$-fold point at $D$ & 6  \\
nodes at $E,F,H, I, J$ & 5 \\
\hline
sum & 20 \\
\hline
\end{tabular}
\end{center}
Therefore the expected dimension of $\sV_U(3)$ is
\begin{align*}
	\expectedDim &= {d-e+3 \choose 2} + {d+1 \choose 2} - (e-1)^2 + \deg X \\
	&= 0+6-5^2+20 = 1
\end{align*}

Let $\omega$ be a differential in $\sV_U(3)$. Consider
the reducible curve $\Gamma =  L_1 \cup L_2 \cup L_3 \cup Q$. By Lemma \ref{lUnionOfIntegralCurves} it
is also an integral curve of $\omega$. Let $K_{L_4}$ and $K_\Gamma$ be the cofactors of
the respective curves. Evaluating at the $4$-fold point $C$ and at the nodes, we obtain the
following matrix of values for $(K_{L_4}:K_\Gamma:d\omega)$
\[
	\begin{pmatrix}
      0 &       1 &	    1\\
      1 &       3 &      2\\
      \end{pmatrix}
      	\begin{tabular}{c}
      (at nodes) \\
      (at $C$)\\
      \end{tabular}
\]
of which some rows may possibly be replaced by zeros.

The kernel of this matrix contains  $(1,-1,1)^t$, so 
\[
	K_{L_4}-K_\Gamma + d\omega
\]
vanishes at $C,E,F,H,I,J$. If these points are in general position with respect to
conics, this implies that 
\[
	K_{L_4}-K_\Gamma + d\omega = 0
\]
as polynomials. By Darboux's Theorem this shows then, that $\omega$ has a rational integrating
factor. 

It is now easy to check that 
\begin{align*}
	\omega = 
	& (-40 x^{2} y-36 x y^{2}-4 y^{3}-80 x^{2}-102 x y-34
      y^{2}-60 x-72 y-40)dx\\
	& +(24 x^{3}+16 x^{2} y+2 x^{2}+12 x y-9 x+2 y-2)dy
\end{align*}
has integral curves 
\begin{align*}
	L_1 &= y+2\\
	L_2 &= 4x+y+4\\
	L_3 &= 4x+1\\
	L_4 &= 2x+1\\
	Q &= 2x^2+2xy+x+2y+2
\end{align*}
in the above configuration, that $\deg X = 20$, that $\dim \sV_U(3) = 1 = \delta$, and that
the points $C,E,F,H,I,J$ do not lie on a conic. Furthermore, one verifies that $\omega$ has only finitely many integral conics (see our script \verb#example_9_6.m2# at \cite{cfGitHub}).

It follows that there exists a $9$-dimensional family of degree $3$ differential 
forms with a rational integrating factor. Over a finite field
one can also check, that the tangent space to the center variety at this point has dimension at most $9$ by computing the derivatives of the first $10$ focal polynomials of $\omega$. This proves that the constructed family is a component of
the center variety.

Differential forms constructed in this way lie in Steiner’s ideal $9.6$ \cite{Ste:2011}.

It seems to us that this component is new.
 \end{constr}

\begin{figure}
\includegraphics*[scale=0.4]{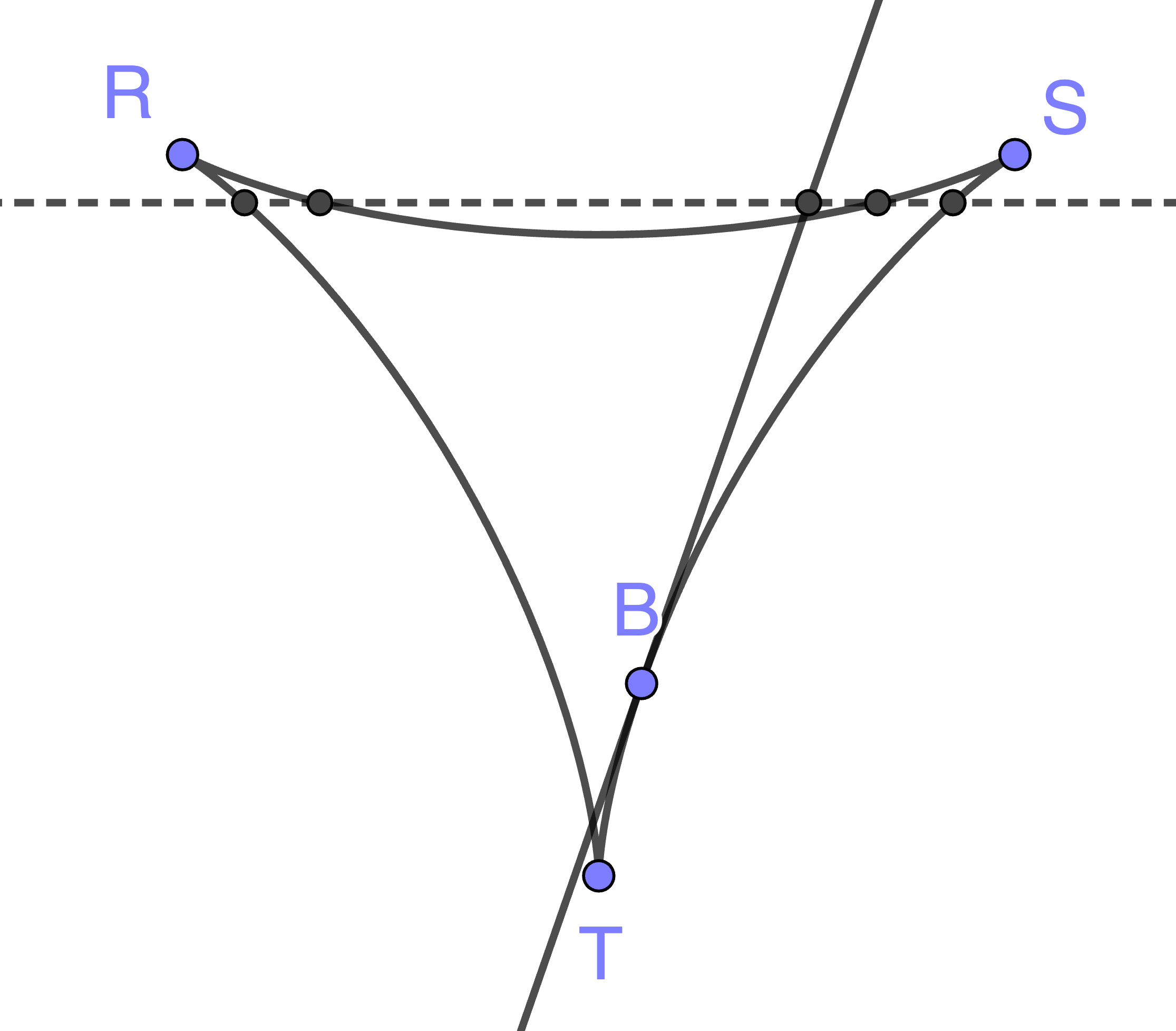}
\caption{Curve configuration for Construction \ref{c98}}
\label{f98}
\end{figure}

\begin{constr} \xlabel{c98}
Let $C$ be a 3-cuspidal curve with nodes labeled $R, S, T$, and let $L$ be a line tangent to $C$ at a point $B$ as considered in Examples \ref{eThreeCuspidalAndLine} and \ref{eThreeCuspidalAndLine2} and Figure \ref{f98}. By 
the following dimension count we, obtain at least a $9$ dimensional family of
such configurations.
\begin{center}
\begin{tabular}{|c|c|}
\hline
plane quartic  & 14 \\
3 cusps            & -6 \\
point B             &   1 \\
tangent to B    &  0\\
\hline
sum & 9\\
\hline
\end{tabular}
\end{center}

The union $U = C \cup L$ of the two curves has degree $e=5$. The degree of $X = V(U,U_x,U_y)$ is $6+3+2=11$ since we have three cusps of $C$, a tacnode at $B$ and
two nodes at the other intersection points.

Therefore, by Propsiiton \ref{pExpected}, the expected dimension of $V_U(3)$ is
\[
	\expectedDim = 0 + 6 - 4^2 + 11 = 1
\]
Let $\omega$ be a differential in $\sV_U(3)$ and
$K_C, K_L$ the corresponding cofactors. We can predict the values of
$(K_L:K_C:d\omega)$ at the special points by Proposition \ref{pSingularPoints}. Since we have
$5$ points on $\lineinfty$ we can also apply Proposition \ref{pManyPointsInfinity} for general points on $\lineinfty$:
\[
	\begin{pmatrix}
      0 &       6 &	    {5}\\
      2 &       2 &       {3}\\
      1 &       4 &       {4}\\
      \end{pmatrix}
      	\begin{tabular}{c}
      (at cusps) \\
      (at $B$)\\
      (at infinity)\\
      \end{tabular}
\]
where some row may be replaced by zero vectors. Since $(4,5,-6)^t$ lies in the kernel of this matrix
\[
	4K_L+5K_C-6d\omega
\]
vanishes at $\lineinfty$ and at the points $R,S,T,B$. If the points $R,S,T,B$ do not lie on a line,
this implies that 
\[
	4K_L+5K_C-6d\omega = 0
\]
as polynomials, and therefore $\omega$ has a rational Darboux integrating factor. The genericity conditions can be verified, for example, for
\begin{align*}
	\omega = &(-11 x^{2} y+3 x y^{2}+y^{3}-2 x^{2}+x y+5 y^{2}-2 x+4 y)dx \\
      &+(11 x^{3}-3 x^{2} y-x y^{2}+10 x^{2}-4 x y-y^{2}-x-y)dy
\end{align*}
with integral curves
\begin{align*}
L &= -7x+2y+1 \\
C &=  x^2y^2-2x^2yw-2xy^2w+x^2w^2-2xyw^2+y^2w^2
\end{align*}
where $w = x+y+1$. This proves that this construction describes a $9$-dimensional component of degree $3$ differential forms with a center. Differential forms constructed in this way lie in Johannes Steiner's experimental ideal 9.8. 

Experimentally, we also observe that differential forms as constructed above automatically contain further integral curves of degree $2$ and $3$ and infinitely many integral curves of degree $6$.

It seems to us that this component is new.

\end{constr}

\begin{figure}
\includegraphics*[scale=0.4]{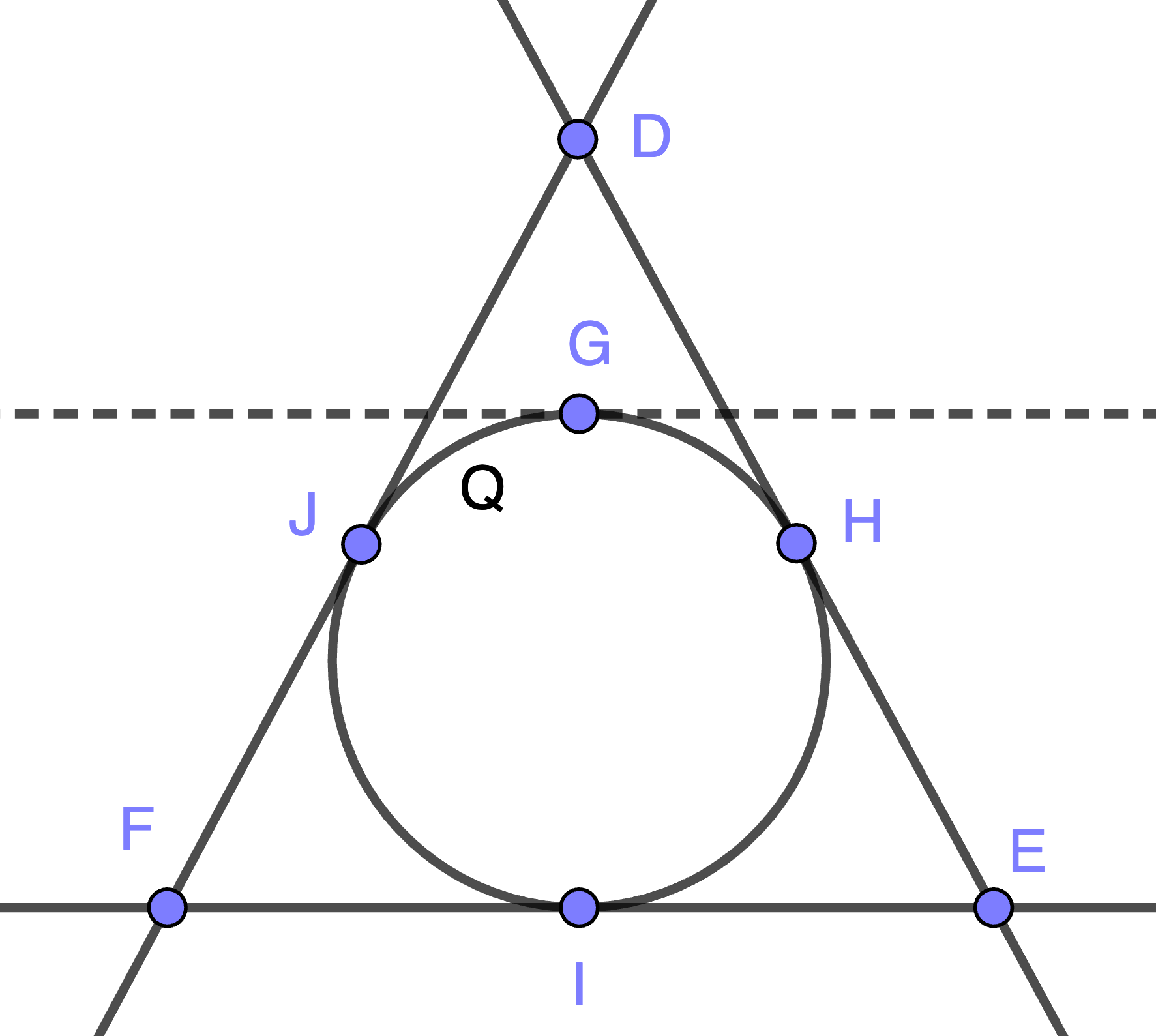}
\caption{Curve configuration for Construction \ref{c99}}
\label{f99}
\end{figure}

\begin{constr} \xlabel{c99}
Choose a point $G$ on $\lineinfty$, and $3$ general points $H,I,J$ in the affine
plane. Let $Q$ be a conic tangent to the line at infinity at $G$, and passing through $H,I,J$. Let
$T_H$, $T_I$ and $T_J$ be the tangent lines to $Q$ in the respective points and $D,E,F$ the intersection points of the tangent lines as shown in Figure \ref{f99}. By the following 
dimension count, we obtain a $7$ dimensional family of such configurations

\begin{center}
\begin{tabular}{|c|c|}
\hline
point G         & 1 \\
points H,I,J  &  6 \\
quadric Q    &   0 \\
\hline
sum & 7\\
\hline
\end{tabular}
\end{center}

Let $U = T_H \cup T_I \cup T_J \cup Q$, $e = \deg U = 5$.
The degree of $X = V(U,U_x,U_y)$ at the special points is:
\begin{center}
\begin{tabular}{|c|c|c|}
\hline
point & $t$ / $t_z$\\
\hline
tangent to $\lineinfty$ at $C$ & 1 \\
tacnodes at $H$,$I$,$J$ & 9  \\
nodes at $E,F,G$ & 3  \\
\hline
sum & 13 \\
\hline
\end{tabular}
\end{center}
Therefore, by Proposition \ref{pExpected}, the expected dimension of $\sV_U(3)$ is
\begin{align*}
	\expectedDim &= 0 + 6 - 4^2 + 13 = 3.
\end{align*}

Let $\omega$ be a differential form in $\sV_U(3)$. Consider
the triangle $T =  T_H \cup T_I \cup T_J$, the inscribed conic $Q$, and their
respective cofactors $K_Q$ and $K_T$. Evaluating at the special points we obtain:
\[
	\begin{pmatrix}
      0 &       1 &	    {1}\\
      2 &       2 &       {3}\\
      \end{pmatrix}
      	\begin{tabular}{c}
      (at nodes) \\
      (at tangent points)\\
      \end{tabular}
\]
where some rows may be replaced by zeros. Since $(1,2,-2)^t$ lies in the kernel of this matrix
\[
	K_Q+2K_T-2d\omega
\]
vanishes at both the tangent points and the nodes. If these $6$ points do not lie on a conic,
this shows that 
\[
	K_Q+2K_T-2d\omega = 0
\]
as polynomials and that $\omega$ has a rational Darboux integrating factor. The  genericity conditions can be verified, for example, with
\begin{align*}
	\omega = &(-2 x y^{2}-2 y^{3}+2 x y+2 y)dx \\
		     &+ (4 x^{3}+2 x^{2} y+xy^{2}-4 x y-y^{2}-x-y)dy
\end{align*}
Since we have a $7$ dimensional family of configurations and a $3$ dimensional space of syzygies for each,
this proves that this construction describes a $7+2=9$ dimensional component of degree $3$ differential forms with a center. Differential forms constructed in this way lie in Johannes Steiner's ideal 9.9. We believe that this component is also new.
\end{constr}

\begin{figure}
\includegraphics*[scale=0.4]{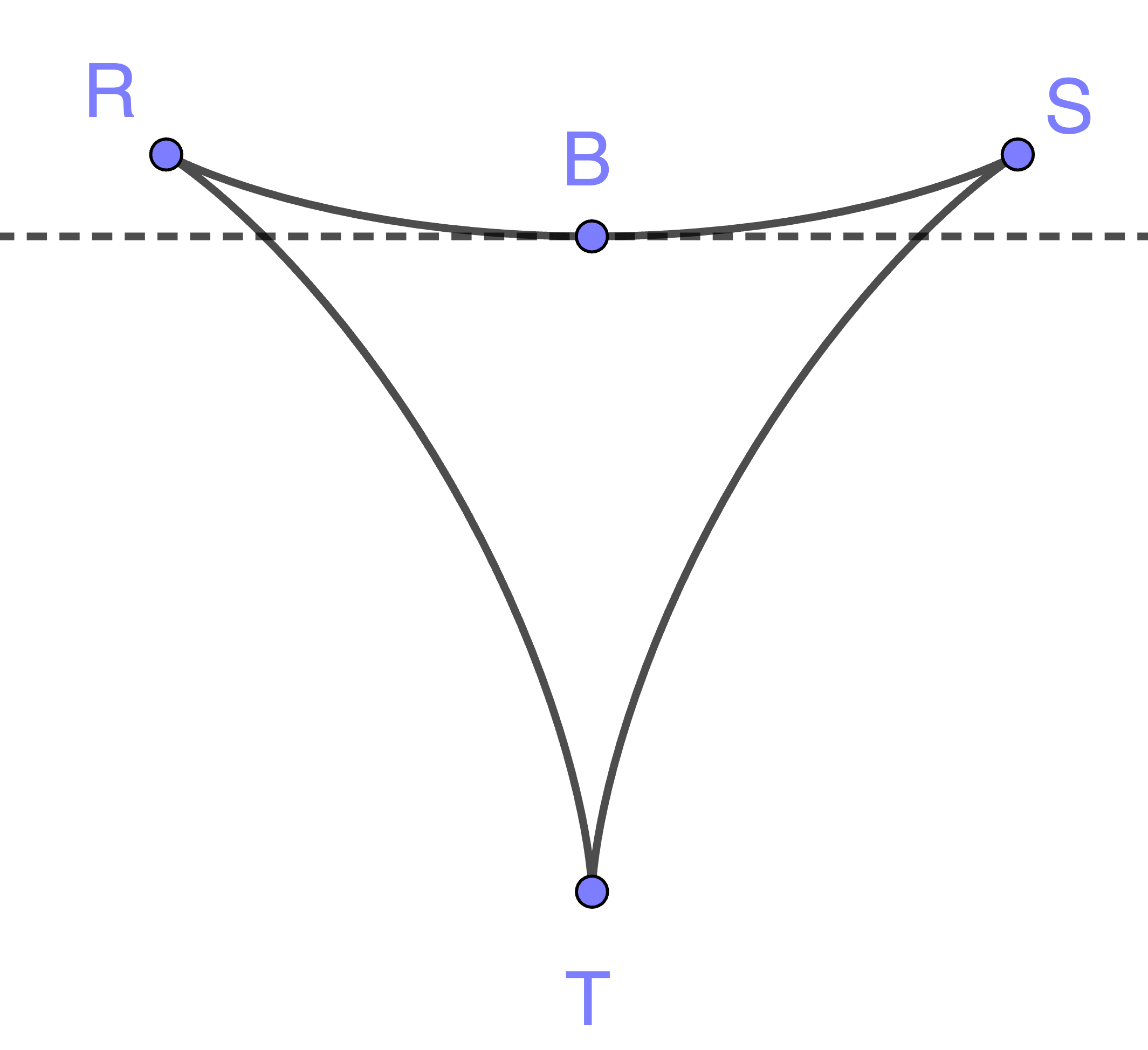}
\caption{Curve configuration for Construction \ref{c910}}
\label{f910}
\end{figure}

\begin{constr} \xlabel{c910}
Consider a $3$-cuspidal quartic curve tangent to $\lineinfty$. Label the cusps with $R,S,T$ and the tangent point with $B$ as in Figure \ref{f910}. There is a $14-3\cdot2-1=7$ dimensional family of such configurations.
The degree of $X$ defined by $(C,C_x,C_y)$ at the special points is

\begin{center}
\begin{tabular}{|c|c|}
\hline
point & t \\
\hline
3 cusps & 6 \\
at $B$ & 1  \\
\hline
sum & 7 \\
\hline
\end{tabular}
\end{center}

Therefore, the dimension of  $V_C(2)$ is 
\begin{align*}
	\expectedDim &= 0 + 3 - 3^2 + 7 = 1
\end{align*}
Let $\omega$ be a differential form in $\sV_C(2)$ and
let $K$ be the cofactor of $C$. Evaluating $(K:d \omega)$ at the cusps yields $(6:5)$. 
Hence, $5K-6d\omega$ vanishes at these points. Now $K$ and $d\omega$ are linear forms
and so $5K-6d\omega = 0$ as polynomials if the cusps do not lie on a line. 
In this case, $\omega$ has a rational Darboux integrating factor $\mu$. 

This family is already known. Differential forms of this family automatically have further integral
curves of degree $2$ and $3$ in special position and coincide with the well-known component of codimension $3$ of degree $2$ differential forms. At most our technique gives a new proof of 
Darboux integrability for differential forms of this type. 

Now consider a general linear form $L$, and the differential form $L\omega$. It will be of
degree $3$ and has an integrating factor $\frac{\mu}{L}$. Since the genericity conditions can be verified for
\begin{align*}
	\omega =  (x+1)\bigl( &(-x^{2}+20 x y+5 y^{2}+x-5 y)dx \\
	&+ (-5 x^{2}-20 x y+y^{2}+5
      x-y)dy\bigr)
\end{align*}
with integral curve
\begin{align*}
C &=  x^2y^2-2x^2yw-2xy^2w+x^2w^2-2xyw^2+y^2w^2
\end{align*}
and $w =  (1/8)(x+y-z)$,
we have a $7$-dimensional family of configurations and a $3$ dimensional family
of degree $3$ syzygies that give rise to integrable degree $3$ differential forms. 
This gives again a $7+2=9$ dimensional component of the center variety.

Differential forms constructed in this way lie in Johannes Steiner's ideal 9.10. 

This component does not appear in \zoladeks lists, but possibly it was left out
because the construction is a trivial extension of a well known degree $2$ component.
 
\end{constr}


\begin{figure}
\includegraphics*[scale=0.4]{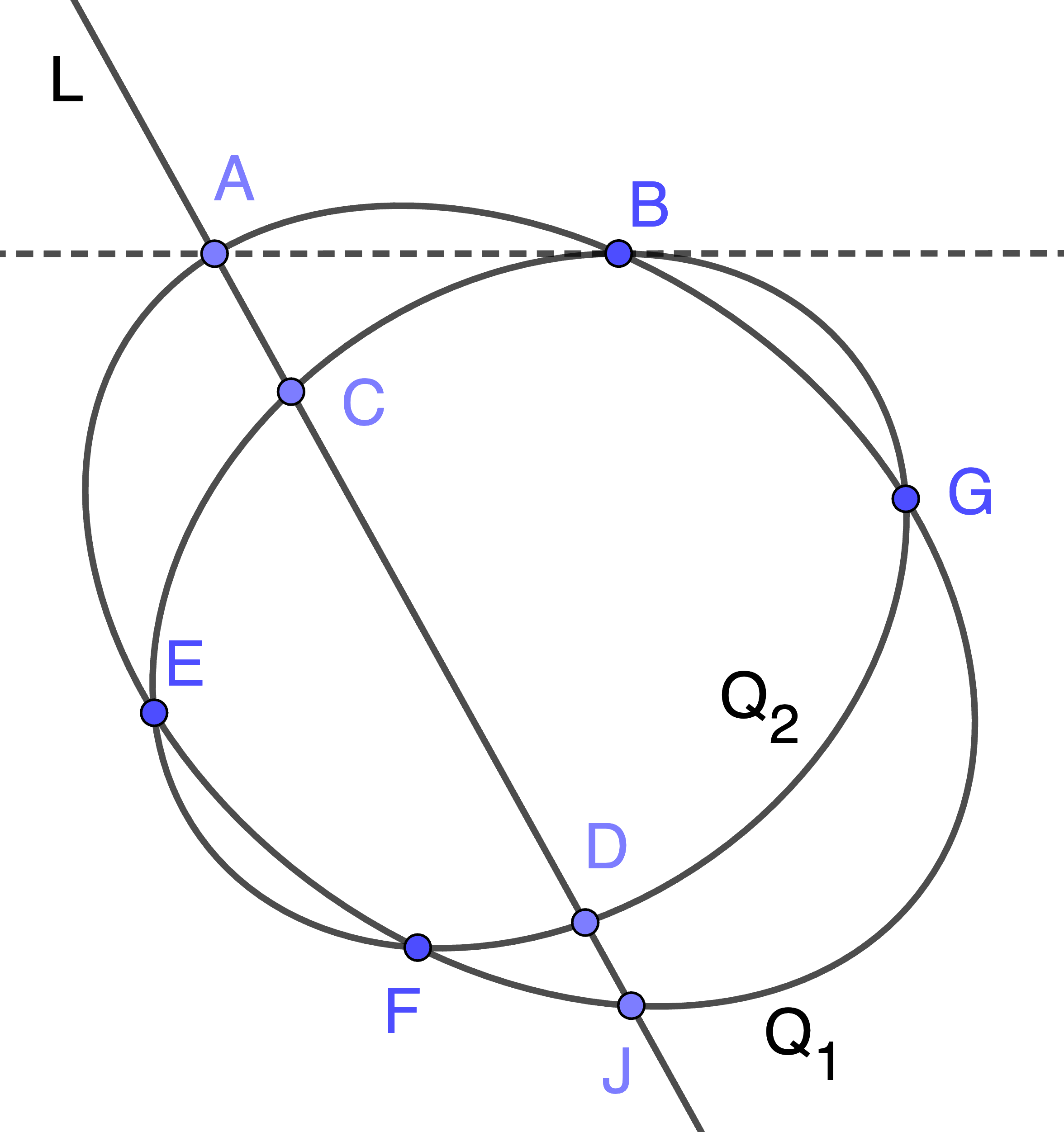}
\caption{Curve configuration used in Construction \ref{c914}}
\label{f914}
\end{figure}

\begin{constr} \xlabel{c914}
Consider a conic $Q_1$ that intersects $\lineinfty$ in points $A$ and $B$.
Choose a second conic $Q_2$ tangent to $\lineinfty$ at $B$. Finally,
choose a line $L$ through $A$ as in Figure \ref{f914}. By 
the following dimension count we have at least a $9$ dimensional family of
such configurations.
\begin{center}
\begin{tabular}{|c|c|}
\hline
conic $Q_1$    	& 5 \\
conic $Q_2$ 	& 3 \\
line L         	& 1 \\
\hline
sum & 9\\
\hline
\end{tabular}
\end{center}

Let $U = Q_1 \cup Q_2 \cup L$ be the union of all curves. It has degree $e =5$.  The degree of 
$X=V(U,U_x,U_y)$ at the special points is

\begin{center}
\begin{tabular}{|c|c|}
\hline
point & $t$ / $t_z$ \\
\hline
$2$-fold point at $A$ & $2$  \\
$2$-fold point with tangent at $B$ & $3$  \\
$6$ nodes & $6$ \\
\hline
sum & 11 \\
\hline
\end{tabular}
\end{center}

Therefore, by Proposition \ref{pExpected}, the expected dimension of $\sV_U(3)$ is
\[
	\expectedDim = 0 + 6 - 4^2 + 11 = 1.
\]
Let now $\omega$ be a differential form in $\sV_U(3)$ and
$K_U$ the cofactor of $U$. The value of  $(K_U:d\omega)$ at the $6$ nodes is $(1:1)$ and therefore 
$K_U-d\omega$ vanishes at all nodes. If the nodes do not lie on a conic this implies
$K_U = d\omega$ and $\omega$ has a rational Darboux integrating factor.

The genericity conditions are satisfied by
\begin{align*}
	\omega = &(-68 x^{3}-168 x^{2} y-216 x y^{2}-64 y^{3}-41 x^{2}-98 x y-16 y^{2}+57 x+46 y)dx \\
      			&+ (-52 x^{3}+32 x^{2} y+32 x y^{2}+20x^{2}+24 x y+80 x+24 y)dy
\end{align*}
with integral curves
\begin{align*}
	L &= x+2y \\
	Q_1 &= 2x^2+4xy+x+1\\
	Q_2  &= 15x^2-19x-4y.
\end{align*}
We obtain a $9$-dimensional component of the center variety in degree $3$. This component is not completely new: it contains \zoladeks family $CD_{28}$ as a subset of codimension $1$. The corresponding differential forms lie in Johannes Steiner’s ideal 9.14.

Experimentally, differential forms from this construction also have another algebraic integral curve of degree $3$ and infinitely many algebraic integral curves of degree $4$.
\end{constr}

Table \ref{tOverview} provides an overview of the current situation in codimension $9$ according to our knowledge. All experimental data used in this study is based on integral curves of degree at most 6, as computed by Johannes Steiner  \cite{Ste:2011}. In several cases (i.e., ideals 9.3, 9.7 and 9.11--9.13), the curves detected within this degree bound were insufficient to construct a rational integrating factor or to verify Darboux integrability. These gaps in the table therefore reflect limitations of the available data and indicate directions for future investigation, possibly involving higher-degree curves.


\begin{table}
\begin{tabular}{|c|c|c|}
\hline
Steiner's experimental ideals & \zoladeks families & This paper \\
\hline
9.1 & $CR_9$ &\\
9.2 & $CD_8$ &\\
9.3 &  &\\
9.4 & $CR_{15}$ &\\
9.5 & $CD_{21}$, $CR_{10}$  &\\
9.6 & & Construction \ref{c96} \\
9.7 & &\\
9.8 &  & Construction \ref{c98} \\
9.9 & & Construction \ref{c99} \\
9.10 &  & Construction \ref{c910} \\
9.11 & $(CD_{30})$  &\\
9.12 &  &\\
9.13 &  &\\
9.14 & $(CD_{28})$  & Construction \ref{c914}\\
\hline
\end{tabular}

\caption{Overview of codimension $9$ components. Numbers in parentheses indicate families that lie within the given component but do not span it entirely.}
\label{tOverview}

\end{table}



\def\cprime{$'$} \def\cprime{$'$}

\end{document}